\documentclass{amsart}
\usepackage{graphicx}
\usepackage{mathptmx}      
\usepackage{graphics}
\usepackage{hyperref}
\usepackage{mathrsfs}
\usepackage[mathcal]{euscript}
\usepackage{amsmath}
\usepackage{mathrsfs}
\usepackage{amssymb}
\usepackage{microtype}
\usepackage{dsfont}

\newcommand{\norm}[1]{\left\Vert#1\right\Vert}

\newcommand{\set}[1]{\left\{#1\right\}}
\newcommand{\Real}{ \mathds{R} }

\newcommand{\Rn}{\Real^n}
\newcommand{\Rnx}{\Real^n_x}
\newcommand{\RnRt}{\Real_x^n\times\Real_t}
\newcommand{\RN}{\Real^N}

\newcommand{\To}{\rightarrow}
\newcommand{\dx}{\,\textrm{d}x}

\newcommand{\dt}{\,\textrm{d}t}
\newcommand{\ds}{\,\textrm{d}s}
\newcommand{\dS}{\,\textrm{d}S}

\newcommand{\bd}{\,\textrm{d}}
\newcommand{\dX}{\,\textrm{d}x\textrm{d}t}
\newcommand{\Lp}[1]{ \mathrm{L}^{#1} }
\newcommand{\Hk}[1]{ \mathrm{H}^{#1} }
\newcommand{\Hkz}[1]{ \mathrm{H}_0^{#1} }

\newcommand{\Grad}{\mathbf{\nabla}}

\newcommand{\Laplace}{\Delta}

\newcommand{\vu}{\mathbf{u}}
\newcommand{\vv}{\mathbf{v}}
\newcommand{\vw}{\mathbf{w}}

\newcommand{\pDom}{ {\Omega_T} }
\newcommand{\eDom}{\Omega}
\newcommand{\B}{\textrm{B}}
\newcommand{\Ball}[2]{\textrm{B}_{#2}(#1)}
\newcommand{\cyl}{Q}

\newcommand{\V}{V}

\newcommand{\vphi}{\phi}
\newcommand{\vpsi}{\psi}
\newcommand{\vPhi}{\Phi}

\newcommand{\Keywords}[1]{\noindent {Keywords}: #1\\} 
\newcommand{\MSC}[1]{\noindent {Mathematics Subject Classification (2000)}: #1\\}
\def\and{$ \cdot \ $}

\def\XXint#1#2#3{{\setbox0=\hbox{$#1{#2#3}{\int}$}
\vcenter{\hbox{$#2#3$}}\kern-.5\wd0}}

\def\Xiint#1{\mathchoice
{\XXiint\displaystyle\textstyle{#1}}
{\XXiint\textstyle\scriptstyle{#1}}
{\XXiint\scriptstyle\scriptscriptstyle{#1}}
{\XXiint\scriptscriptstyle\scriptscriptstyle{#1}}
\!\iint}
\def\XXiint#1#2#3{{\setbox0=\hbox{$#1{#2#3}{\iint}$}
\vcenter{\hbox{$#2#3$}}\kern-.5\wd0}}

\def\dashiint{\Xiint{\rule{1.2em}{0.5pt}}}
\def\lequiv{:=}
\def\u{u}
\def\v{v}
\def\w{w}

\numberwithin{equation}{section}

\theoremstyle{plain}
\newtheorem{theorem}{Theorem}[section]
\newtheorem{lemma}[theorem]{Lemma}

\theoremstyle{definition}

\newtheorem{remark}[theorem]{Remark}

\begin{document}
\title{Everywhere regularity of certain types of parabolic systems}
\author{Maxim Trokhimtchouk$^1$}

\thanks{$^1$Department of Mathematics, University of California, Berkeley California USA 94720. }
\thanks{The author would like to thank his thesis advisor L. C. Evans for his continuous support and helpful advise.}

\begin{abstract}
In this paper I discuss nonlinear parabolic systems that are generalizations
of scalar diffusion equations. More precisely, I consider systems of the form
$$
\vu_t -\Laplace\left[ \Grad\vPhi(\vu)\right] = 0 ,
$$
where $\vPhi(z)$ is a strictly convex function. I show that when $\vPhi$ is a function only of the norm of $\vu$, then bounded weak solutions of these parabolic systems are everywhere H\"older continuous and thus everywhere smooth. I also show that the method used to prove this result can be easily adopted to simplify the proof of the result due to Wiegner ~\cite{Wiegner92} on everywhere regularity of bounded weak solutions of strongly coupled parabolic systems. \\
\end{abstract}

\maketitle 
\MSC{35K40 \and 35K55 \and 35B65 \and 35B35 \and 35D10}
\tableofcontents

\section{Introduction}
The theory of regularity of nonlinear \emph{scalar} elliptic and parabolic
equations is by now classical. It goes back to the ground breaking work on
equations in divergence form by De Giorgi, Moser and Nash in the late fifties.
Since then, H\"older estimates for general nonlinear elliptic and parabolic
equations were derived by Krylov and Safonov. However, it became clear quite early on, after discovery of counter examples, that nonlinear elliptic and parabolic \emph{systems} do not, in general, possess \emph{everywhere regularity}. Instead, only \emph{partial regularity} results are available
~\cite{Giaquinta83}.

Despite the lack of everywhere H\"older continuity of weak solutions of general elliptic and parabolic systems, there are several nontrivial examples that do possess everywhere regularity due to their special structure. One of the earliest such examples of a fully nonlinear elliptic system whose weak solutions are everywhere H\"older continuous is due to Uhlenbeck ~\cite{Uhlenbeck77}. She considered elliptic systems of the form
$$
-(g (|\Grad \vu|^2)\u^i_{x_\alpha})_{x_\alpha} = 0 \quad \textrm{for }
i\in \set{1\ldots N},
$$
with some additional ellipticity and grows conditions on $g$. Uhlenbeck's
proof of H\"older continuity of weak solutions of these systems relies crucially on the existence of an auxiliary function which is subharmonic. This auxiliary function (\emph{entropy}) is tied to $|\Grad^2 \vu|^2$ by an inequality, and this allows local control on the second derivatives of $\vu$.

The parabolic examples of fully nonlinear systems possessing everywhere regularity followed somewhat later. Wiegner in ~\cite{Wiegner92} provided an original example of such a system. He showed everywhere H\"older continuity of weak solutions of \textbf{strongly coupled systems}
\begin{equation}\label{eq:strongly_coupled_begin}
\u^i_t - (a_{\alpha\beta}\u^i_{x_\beta} +
c^i_{\alpha\beta}H_{x_\beta})_{x_\alpha} = 0,
\end{equation}
with $H\lequiv H(\vu)$ a strictly convex function of $\vu$ and strict ellipticity conditions on the coefficients. His work has been followed by several others,
including Dung ~\cite{Dung99}, on strongly coupled system, and  Bae and
Choe ~\cite{BaeChoe04} on a parabolic analog of the example of Uhlenbeck. Unlike Uhlenbeck's work, however, none of the proofs of everywhere regularity for aforementioned parabolic examples rely explicitly on the existence of entropy.

In this paper I provide new examples of parabolic systems whose weak
solutions are everywhere H\"older continuous. These systems are a type of \textbf{nonlinear diffusion systems}, and are of the form
\begin{equation}\label{eq:diffusion_system}
\vu_t -\Laplace\left( \Grad\vPhi(\vu)\right) = 0
\end{equation}
where $\vPhi(z)$ is a strictly convex function. These systems are interesting since they generalize \textbf{scalar nonlinear diffusion} equations of the form
$$
\u_t - \Laplace ( \gamma(u)) = 0,
$$
where $\gamma$ is strictly increasing. I will show that if $\vPhi(z)$
depends only on the norm of $z$, then, together with some smoothness and
grows conditions on $\vPhi$, weak solutions of diffusion systems
\eqref{eq:diffusion_system} are everywhere H\"older continuous. I will do this by exhibiting an entropy, and showing that existence of entropy together with properties of general parabolic systems suffices to prove everywhere regularity. In the conclusion, I will show that regularity of bounded weak solutions to the strongly coupled systems \eqref{eq:strongly_coupled_begin} can also be obtained with the help of an entropy.


 Let me introduce notation that I will use throughout this
paper. I consider parabolic space to consist of space and time with space
being of $n$ dimensions and time of one dimension. In this paper I will deal
with cylindrical domains for simplicity. By elliptic domain $\eDom\subset \Rnx$ I
will mean space domain. Cylindrical parabolic domain $\eDom\times (0,
T)\subset \RnRt$, where $T>0$, will be denoted by $\pDom$. The time
interval of the cylindrical domain $\pDom$, that is the interval $(0,T)$ in
the case of $\pDom = \eDom\times (0,T)$, will be denoted by $I(\pDom)$.
By $\Ball{x}{r}$ I will mean an $n$ dimensional ball as a subset of $\Rnx$ with center at $x$ and
of radius $r$. $\cyl(x,t,r)$ will denote the cylinder
$$
\cyl(x,t,r)\lequiv\Ball{x}{r}\times (t-r^2,t).
$$
I will write $\B_r$ or $\B$ instead of
$\Ball{x}{r}$ when $x$ or $r$ are clear from the context. Similarly,
sometimes I will write $\cyl_r$ or $\cyl$ instead of $\cyl(x,t,r)$. For a
Lebesgue measurable set $S$ by $|S|$ I will mean Lebesgue measure of
$S$. Finally $V(\pDom; \Rn)$ will denote the closure of $C^1(\pDom; \Rn)$
functions under the norm
\begin{equation*}
\norm{v}_{V(\pDom)}^2 = \sup_{t\in I(\pDom)} \int_{\eDom}|v(x,t)|^2\dx +
\iint_{\pDom} |\Grad v(x,t)|^2\dx\dt.
\end{equation*}

\section{Partial regularity of nonlinear systems}
In this section we recall partial regularity results for weak
solutions of quasi-linear parabolic systems of the form
\begin{equation}\label{eq:nonlinear_system}
\u^i_t - (A^{\alpha\beta}_{ij}(x, \vu) \u^j_{x_\beta})_{x_\alpha}=0 \quad
\forall i \in \set{1\ldots N},
\end{equation}
where coefficients $A^{\alpha\beta}_{ij}$ satisfy  the following condition of
\textbf{strong ellipticity}:
\begin{equation}\label{eq:strong_ellipticity}
A^{\alpha\beta}_{ij}\xi^i_\alpha \xi^j_\beta \ge \lambda|\xi|^2
\textrm{ for all }\xi\in \Real^{nN}, \textrm{ for some }\lambda>0.
\end{equation}
By a weak solution in this case I mean a function $\vu\in
V(\pDom; \RN)$ that satisfies
$$
\iint_\pDom - \u^i\v^i_t + A^{\alpha\beta}_{ij} \u^j_{x_\beta}
\v^i_{x_\alpha}\dX = 0, \quad \textrm{for all }\vv\in \Hkz{1}(\pDom;\RN).
$$
It is well known that weak solutions of systems of this type possess partial
regularity under some appropriate continuity conditions on
$A^{\alpha\beta}_{ij}$. The important result in this area is the following
local regularity result due to Giaquinta and Struwe:
~\cite{GiaquintaStruwe82}
\begin{theorem}[Local regularity condition]\label{thm:local_regularity}
Suppose coefficients $A^{\alpha\beta}_{ij}$ satisfy condition
\eqref{eq:strong_ellipticity}, are continuous and bounded. Also
suppose $\vu\in \V(\pDom)$ is a weak solution of
\eqref{eq:nonlinear_system}. Then, if for some $(x_0, t_0)\in \pDom$
\begin{equation}\label{eq:local_regularity_condition}
\liminf_{R\To 0}\frac{1}{R^n}\iint_{Q(x_0, t_0 ,R)} |\Grad \vu|^2 \dX= 0,
\end{equation}
then $\vu$ is H\"older continuous in the neighborhood of $(x_0, t_0)$.
\end{theorem}
Condition \eqref{eq:local_regularity_condition} is the basis for the proofs of
everywhere regularity that we will discuss in the rest of this paper. We will, however, need one more result due to Giaquinta and Struwe ~\cite{GiaquintaStruwe82}.
\begin{lemma}[$\Lp{p}$ estimate]\label{lem:lp_estimate}
Let $\vu$ be a weak solution of the system \eqref{eq:nonlinear_system}. Then there exists an exponent $p>2$ such that $|\Grad \vu|\in \Lp{p}_{loc}(\pDom)$; moreover for all $Q_R\subset Q_{4R} \subset \pDom$ we have
\begin{equation}\label{eq:Lp_estimate}
\left(\dashiint_{Q_R} |\Grad \vu|^p\dX\right)^{1/p}\leq C \left(\dashiint_{Q_{4R}} |\Grad \vu|^2\dX\right)^{1/2}.
\end{equation}
\end{lemma}

\section{Generalized diffusion equations}
In this section I discuss the type of parabolic systems I will refer to as
\textbf{diffusion system}. Let $\vPhi:\RN\To\Real$ be a strictly convex, twice continuously differentiable function with
\begin{equation}\label{eq:strict_convexity}
\lambda |\xi|^2 \le \vPhi_{z_i z_j} \xi^i \xi^j \leq \Lambda |\xi|^2.
\end{equation}
Then we say that $\vu$ is a weak solution of
\begin{equation}\label{eq:gen_diff_eq}
\vu_t - \Laplace\left[\Grad\vPhi(\vu)\right] = 0
\end{equation}
if $\vu\in \V(\pDom; \RN)$ and for
all $\vw \in \Hkz{1}(\pDom; \RN)$
\begin{equation}
\iint_\pDom -\u^i\w^i_t + (\vPhi_{z_i}(\vu))_{x_\alpha}
\w^i_{x_\alpha} \dX = 0.
\end{equation}
This is a standard quasi-linear elliptic system of the type in
(\ref{eq:nonlinear_system}), since we can rewrite it as
\begin{equation}
\u^i_t + \left(\vPhi_{z_i z_j}(\vu)\u^j_{x_\alpha}\right)_{x_\alpha} = 0
\quad \textrm{for all }i\in \set{1\ldots N}.
\end{equation}
If in addition $\vu$ is bounded I say $\vu$ is bounded weak solution.

As I have mentioned before this equation is a generalization of scalar nonlinear diffusion equation. It has several nice properties. First, its flow is a contraction in $\Hk{-1}(\eDom)$ and its solutions are unique. Second, weak solutions of \eqref{eq:gen_diff_eq} are in $\Hk{1}_{loc}(0,T; \Lp{2}_{loc}(\eDom))$, that is their weak derivatives in time are in $\Lp{2}_{loc}(\pDom)$. Furthermore, if the solution is bounded, the gradient in $x$ is actually in $\Lp{4}_{loc}(\pDom)$.

First I show that flow is a contraction in $\Hk{-1}(\eDom)$.
\begin{theorem}[Uniqueness]\label{thm:flow_contraction}
Let $\u_0, \u_1\in \V(\pDom; \RN)$ be two weak solutions of \eqref{eq:gen_diff_eq} with the same boundary conditions, that is $u_0(\cdot,t)\equiv u_1(\cdot,t)$ on $\partial \eDom$ for almost all $t\in [0,T]$. Denote by $i:\Lp{2}(\eDom) \To \Hk{-1}(\eDom)$ the natural embedding of square integrable functions in $\Hk{-1}$ defined by
$$
i(f)(\phi) := \int_\eDom f\phi \dx.
$$
Then we have for $T \geq t_1 \geq t_0 \geq 0$
\begin{equation}\label{eq:contraction}
\norm{i(\u_0(t_1)-\u_1(t_1))}_{\Hk{-1}(\eDom)}\leq e^{\lambda(t_1-t_0)}\norm{i(\u_0(t_0)-\u_1(t_0))}_{\Hk{-1}(\eDom)},
\end{equation}
where $\lambda$ as in \eqref{eq:strict_convexity}.
\end{theorem}
\begin{proof}
Let us denote by $f_h$ the Steklov average of $f$ defined as
$$
f_h(x,t):=\int_t^{t+h} f(x,s)\ds.
$$
Also for simplicity we will write $\vv_k$ for $\Grad\vPhi(\vu_k)$ with $k=1,2$. It is not very hard to show that $(\vu_k)_h$ and $(\vv_k)_h$ weakly satisfy
\begin{equation}\label{eq:milified_eq}
((\vu_k)_h)_t - \Laplace((\vv_k)_h)=0.
\end{equation}
Let us denote the solution of
$$
\Laplace w = f, \quad w\equiv 0\textrm{ on }\eDom
$$
by $\Laplace^{-1}f$. Then the $\Hk{-1}(\eDom)$ norm of $i(f)$ is given by
$$
\norm{i(f)}_{\Hk{-1}(\eDom)}^2 = \int_\eDom |\Grad (\Laplace^{-1} f)|^2\dx.
$$
For simplicity for an $f\in \Lp{2}(\eDom)$ let us write $\norm{f}_{\Hk{-1}}$ instead of $\norm{i(f)}_{\Hk{-1}(\eDom)}$. Now fix $h$. For $t_0, t_1 \in (h, T-h)$ we compute
\begin{equation}\label{eq:uniquenes_mollify}
\left(e^{2\lambda t}\norm{(\vu_1)_h - (\vu_0)_h}^2_{\Hk{-1}}\right)\Big|_{t_0}^{t_1}   = I_1 + I_2,
\end{equation}
where
\begin{equation*}
I_1 = \int_{t_0}^{t_1}2\lambda e^{2\lambda t}\norm{(\vu_1)_h - (\vu_0)_h}^2_{\Hk{-1}}\Big|_{t}dt,
\end{equation*}
and
\begin{equation*}
I_2 = \int_{t_0}^{t_1} \int_\eDom 2 e^{2\lambda t}\Grad\Laplace^{-1}( (\vu_1)_h - (\vu_0)_h)\cdot \Grad\Laplace^{-1}( (\vu_1)_h - (\vu_0)_h)_t \dX.
\end{equation*}
Using the equation and the fact that $\vu_0$ and $\vu_1$ have the same trace, for $I_2$ we obtain
\begin{equation*}
I_2  =  -\int_{t_0}^{t_1} \int_\eDom 2 e^{2\lambda t} ( (\vu_1)_h - (\vu_0)_h)\cdot ( (\vv_1)_h - (\vv_0)_h)\dX
\end{equation*}
Taking the limit of both sides of the equation \eqref{eq:uniquenes_mollify} as $h\To 0$ we get
\begin{equation}
\left(e^{2\lambda t}\norm{(\vu_1 - \vu_0}^2_{\Hk{-1}}\right)\Big|_{t_0}^{t_1} = I_3 + I_4,
\end{equation}
where
\begin{equation*}
I_3 = \int_{t_0}^{t_1}2\lambda e^{2\lambda t}\norm{\vu_1 - \vu_0}^2_{\Hk{-1}}\Big|_{t}dt,
\end{equation*}
and
\begin{eqnarray*}
I_4 & = & -\int_{t_0}^{t_1} \int_\eDom 2 e^{2\lambda t} ( \vu_1 - \vu_0)\cdot ( \vv_1 - \vv_0)\dX\\
& \leq & -\int_{t_0}^{t_1} \int_\eDom 2\lambda e^{2\lambda t} |\vu_1 - \vu_0|^2\dX\\
& \leq & -\int_{t_0}^{t_1} 2 \lambda e^{2\lambda t} \norm{\vu_1 - \vu_0}^2_{\Hk{-1}}\Big|_{t}dt.
\end{eqnarray*}
Thus we see that $I_3+I_4 \leq 0$ and we establish the result for $T > t_1 \geq t_0 > 0$. We establish the Theorem by taking the limit as $t_0\To 0$ and $t_1\To T$.
\qed\end{proof}
Let us now derive the $\Hk{2}$ estimates for the solutions of \eqref{eq:gen_diff_eq}. In general, quasi-linear parabolic systems do not have $\Hk{2}$ estimates and the fact that solutions of generalized diffusion equations do have them is the consequence of the special structure that equation \eqref{eq:gen_diff_eq} possesses.
\begin{theorem}[$\Hk{2}$ estimates]\label{H2_estimate}
Let $\u\in \V(\pDom; \RN)$ be a weak solution of \eqref{eq:gen_diff_eq}. Then $\u_t\in \Lp{2}_{loc}(\pDom)$ and for $\cyl(x,t,r)\subsetneq \cyl(x,t,R) \subset \pDom$ we have the following estimates
\begin{equation}\label{eq:H2_estimate}
\iint_{\cyl(x,t,r)} |\vu_t|^2\dX 
\leq \frac{C}{(R-r)^2} \iint_{\cyl(x,t,R)} |\Grad\vu|^2 \dX,
\end{equation}
\begin{equation}\label{eq:Hessian_estimate}
\iint_{\cyl(x,t,r)} |\Grad^2 (\Grad_z\vPhi(\vu))|^2\dX \leq \frac{C}{(R-r)^2} \iint_{\cyl(x,t,R)} |\Grad\vu|^2 \dX.
\end{equation}
\end{theorem}
\begin{proof}
As before we will write $\vv$ for $\Grad\vPhi(\vu)$ and $f_h$ for Steklov average of $f$. Denote by $\xi\in C^\infty_0(\eDom)$ a smooth bump function supported in $\Ball{x}{R}\subset\eDom$, which is identically one on $\Ball{x}{r}\subset\eDom$ with $\norm{\Grad\xi}_\infty \leq C/(R-r)$. Also denote by $\eta\in C^\infty_0((t-R^2,t])$ a function that is identically one on $[t-r^2,t]$ and supported in $[t-R^2,t]$ with $|\eta'|\leq C/(R^2-r^2)$. Then multiplying equation
$$
(\vu_h)_t - \Laplace (\vv_h) = 0
$$
by $(\vv_h)_t\xi^2\eta$ we obtain
$$
\int_{t-R^2}^t \int_\eDom (\vu_h)_t(\vv_h)_t\xi^2\eta + \Grad (\vv_h)\Grad(\vv_h)_t\xi^2\eta
+ \Grad (\vv_h)(\vv_h)_t 2\xi\Grad\xi\eta \dX = 0.
$$
Using strict convexity on the first term, integrating the second term in time by parts and using H\"older inequality twice we obtain
\begin{eqnarray*}
\int_{t-R^2}^t \int_\eDom \frac{\lambda}{2}|(\vu_h)_t|^2\xi^2\eta \dX & + & \int_\eDom |\Grad(\vv_h)|^2\xi^2\eta \Big|_t \dx \\
&\leq & C (\norm{\eta'}_\infty + \norm{\Grad\xi}_\infty^2)\int_{t-R^2}^t \int_\eDom |\Grad(\vv_h)|^2\xi^2\dx
\end{eqnarray*}
Thus taking a limit as $h\To 0$ we deduce that $u_t\in L^2_{loc}(\Omega)$ and derive estimate \eqref{eq:H2_estimate} as claimed.
To prove \eqref{eq:Hessian_estimate}, we use $\Hk{2}$ estimates for the Laplacian to get the following for a fixed $t$:
$$
\int_{\Ball{x}{r}\times\set{t}} |\Grad^2 \vv|^2\dx \leq \frac{C}{(R-r)^2} \int_{\Ball{x}{R}\times\set{t}} |\Grad\vv|^2 \dx+ \int_{\Ball{x}{R}\times\set{t}} |u_t|^2\dx.
$$
Integrating in time and using estimate \eqref{eq:H2_estimate} we get
\begin{equation*}
\iint_{\cyl(x,t,r)} |\Grad^2 \vv|^2\dX \leq \frac{C}{(R-r)^2} \iint_{\cyl(x,t,R)} |\Grad\vv|^2 \dX.
\end{equation*}
\qed\end{proof}

In addition to $\Hk{2}$ estimates for weak solutions, bounded weak solutions of equation \eqref{eq:gen_diff_eq} are actually in $\Lp{4}_{loc}$. This is rather unusual since most quasi-linear equations do not have $\Lp{4}$ estimates.
\begin{theorem}[$\Lp{4}$ estimate for bounded solutions]\label{thm:lp_bound_est}
Let $\vu$ be a weak bounded solution of the equation \eqref{eq:gen_diff_eq}. Then  $\Grad \vu$ is locally in $\Lp{4}$ and for $\cyl(x,t,r)\subsetneq \cyl(x,t,R)\subset \pDom$ we have the following estimate:
\begin{equation}\label{eq:L4_estimate}
\iint_{\cyl(x,t,r)} |\Grad\vu|^4\dX \leq
\frac{C\norm{\vu}_\infty^2}{(R-r)^2}\iint_{\cyl(x,t,R)} |\Grad\vu|^2\dX.
\end{equation}
\end{theorem}
\begin{proof}
Let us again denote $\Grad \Phi(\vu)$ by $\vv$ and let $\tau:\Real \To\Real$ be a smooth increasing function that is linear on $(-\infty, 1]$ and constant on $[2, \infty)$. For some large enough constant $C_1$, $C_1\tau(z)\geq z (\tau'(z))^2$. Define $\tau_\epsilon$ as $\tau_\epsilon(x):= \tau(\epsilon x)/\epsilon$. Notice that $\norm{\tau_\epsilon'}_\infty \leq C_0$ and $C_1\tau_\epsilon(z)\geq z (\tau_\epsilon'(z))^2$ with constants independent of $\epsilon$. Letting $\xi$ be a smooth bump function as in the proof of Theorem above, multiply the equation \eqref{eq:gen_diff_eq} by $\vv\tau_\epsilon(|\Grad\vv|^2)\xi^2$ and integrate by parts to obtain
\begin{multline*}
\iint_{\cyl(x,t,R)} \u^i_t\v^i\tau_\epsilon(|\Grad\vv|^2)\xi^2 + |\Grad\vv|^2\tau_\epsilon(|\Grad\vv|^2)\xi^2 \\
+ 2\v^i\v^i_{x_\alpha}\v^j_{x_\beta}\tau_\epsilon'(|\Grad\vv|^2)\v^j_{x_\alpha x_\beta}\xi^2 + 2\v^i\v^i_{x_\alpha}\tau_\epsilon(|\Grad \vv|^2)\xi\xi_{x_\alpha}\dX=0.
\end{multline*}
Hence we compute
\begin{eqnarray*}
\iint_{\cyl(x,t,R)} |\Grad\vv|^2\tau_\epsilon(|\Grad\vv|^2)\xi^2\dX  \leq  \iint_{\cyl(x,t,R)} C_2\norm{\vv}_\infty^2|\vu_t|^2\xi^2 + \frac{1}{2}\tau^2_\epsilon(|\Grad\vv|^2)\xi^2\\
+ C_3\norm{\Grad\xi}^2_\infty \norm{\vv}_\infty^2 |\Grad\vv|^2
+ C_1\norm{\vv}^2_\infty|\Grad^2\vv|^2 \xi^2 + \frac{1}{4C_1}|\Grad\vv|^4 \tau_\epsilon'(|\Grad\vv|^2)^2\xi^2 \dX.
\end{eqnarray*}
Simplifying and using estimates \eqref{eq:H2_estimate}, \eqref{eq:Hessian_estimate} and $C_1\tau_\epsilon(z)\geq z (\tau_\epsilon'(z))^2$ we deduce
$$
\iint_{\cyl(x,t,r)} |\Grad\vv|^2\tau_\epsilon(|\Grad\vv|^2)\dX \leq
\frac{C\norm{\vu}_\infty^2}{(R-r)^2}\iint_{\cyl(x,t,R)} |\Grad\vu|^2\dX.
$$
Taking a limit as $\epsilon\To 0$ we deduce by monotone convergence theorem that $\Grad\vu$ is locally in $\Lp{4}$ and the estimate as claimed in the statement of the Theorem.
\qed\end{proof}
\begin{remark}
The fact that the gradient is actually locally in $\Lp{4}$ implies that the singular set of a bounded solution has parabolic Hausdorff dimension smaller than $n-2$. As you may recall from Theorem \ref{thm:local_regularity} the singular set is contained in the set
$$
\set{(x,t)\in \pDom \ \Big| \ \liminf_{R\To 0}\frac{1}{R^n}\iint_{Q(x, t ,R)} |\Grad \vu|^2\dX > 0 }.
$$
Now using the H\"older inequality and the fact that $\vu$ locally in $\Lp{4}$ we compute
\begin{eqnarray*}
\frac{1}{R^n}\iint_{\cyl(x,t,R)} |\Grad\vu|^2 \dX & \leq &
\frac{1}{R^n}\left( \iint_{\cyl(x,t,R)} |\Grad\vu|^4\dX\right)^{1/2}|\cyl(x,t,R)|^{1/2}\\
&\leq& C R^{\frac{2-n}{2}}\left(\iint_{\cyl(x,t,R)} |\Grad\vu|^4 \dX\right)^{1/2}\\
&\leq& C \left(R^{2-n}\iint_{\cyl(x,t,R)} |\Grad\vu|^4 \dX\right)^{1/2}.
\end{eqnarray*}
Hence the singular set must be contained in the set
$$
\set{(x,t)\in \pDom \ \Big| \ \liminf_{R\To 0}\frac{1}{R^{n-2}}\iint_{Q(x, t ,R)} |\Grad \vu|^4 \dX > 0 },
$$
which has $n-2$ parabolic Hausdorff measure zero.
\end{remark}

\section{The key lemma}
In this section I discuss the parabolic version of the lemma that seems to be
key in the proof of everywhere regularity of some elliptic systems. Not
surprisingly, it will turn out that the parabolic lemma is crucial to proving
everywhere regularity of solutions to some types of parabolic systems. The
elliptic lemma, to which I refer, is well known and the proof of it can be found
in ~\cite{Giaquinta83} in Chapter 7 as part of Theorem 1.1.

Before we proceed with the discussion of this elliptic lemma let us recall that
coefficients $a_{\alpha\beta}$ are called \textbf{strictly elliptic} if there
exists $\lambda >0$ such that
$$
a_{\alpha\beta}\xi^\alpha\xi^\beta \ge \lambda |\xi|^2\quad\textrm{for all }\xi\in \Rn.
$$
\begin{lemma}\label{lem:elliptic_key_lemma}
Suppose coefficients $a_{\alpha\beta}(x)$ are strictly elliptic, bounded and measurable. Let $u\in \V(\eDom)$, $f\in \Lp{1}(\eDom)$ be nonnegative functions satisfying
\begin{equation}\label{eq:elliptic_key_lemma_key_pdi}
- (a_{\alpha\beta}u_{x_\beta})_{x_\alpha} + f \le 0
\end{equation}
on $\eDom$. For any $x_0\in\eDom$ for which
$\Ball{x_0}{R_0}\subset \eDom$ for some $R_0$, we have the
following:
\begin{equation}\label{eq:elliptic_key_lemma_Morrey_decay}
\liminf_{R\To 0} \frac{1}{R^{n-2}}\int_{\Ball{x_0}{R}} f\dx  = 0.
\end{equation}
\end{lemma}

The proof of this lemma is rather simple and follows easily from the elliptic
Harnack inequality. Because of the peculiar geometry of the parabolic
Harnack inequality, the elliptic proof does not translate directly into the parabolic case. Instead, by adopting proof of elliptic lemma to parabolic equations, one is able to control $f$ on the cylinders whose top centers are slightly shifted back in time. In fact it is not true that $f$ can be controlled on the cylinders whose top centers are not fixed, without additional assumptions on $f$. It turns out, however, that the assumption that one needs to impose to prove the parabolic version of the lemma is satisfied in applications to everywhere regularity of parabolic systems. What we need to assume is that for some $\alpha > 1$ the $\Lp{\alpha}$ norm of $f$ on a cylinder is controlled by the $\Lp{1}$ norm of $f$, perhaps on a larger cylinder.

\begin{lemma}[Key Lemma]\label{lem:key_lemma}
Suppose coefficients $a_{\alpha\beta}(x,t)$ are strictly
elliptic, bounded and measurable.
Let $u\in \V(\pDom)$, $f\in \Lp{1}(\pDom)$ be nonnegative
functions weakly satisfying
\begin{equation}\label{eq:key_lemma_key_pdi}
u_t - (a_{\alpha\beta}u_{x_\beta})_{x_\alpha} + f \le 0
\end{equation}
on $\pDom$.
Further suppose that for some $\alpha >1$ our $f$ satisfies the following:
\begin{equation}\label{eq:higher_Lp}
\left(\dashiint_{Q_R} f^\alpha \dX\right)^{1/\alpha}\leq C \ \dashiint_{Q_{4R}} f \dX,
\end{equation}
for all $Q_R\subset Q_{4R}\subset\pDom$. Then for any $(x_0, t_0)\in\pDom$ for which
$$
\B_{R_0}(x_0)\times (t_0 - R_0^2, t_0+R_0^2)\subset \pDom, \textrm{ for some } R_0,
$$
we have the following:
\begin{equation}\label{eq:key_lemma_Morrey_decay}
\liminf_{R\To 0} \frac{1}{R^{n}}\iint_{Q(x_0,t_0,R)} f \dX  = 0.
\end{equation}
\end{lemma}
\begin{proof}
Fix $0 <\sigma \le 1/4$. Set
$$
R_i\lequiv \sigma^i R_0, \quad Q_i\lequiv Q(x_0, t_0, R_i), \quad M_i\lequiv \sup_{Q_i} u
$$
and
$$
Q'_i\lequiv Q(x_0,t_0 - 4\sigma^2 R_i^2, (1-8\sigma^2)^{1/2}R_i),
$$
$$
Q''_i\lequiv Q(x_0,t_0 - 2\sigma^2 R_i^2, (1-4\sigma^2)^{1/2}R_i).
$$
We divide the proof in three steps. In the first step we show that for any $\sigma \in (0, 1/4]$ we have
$$
\lim_{i\To \infty}\frac{1}{R_i^n}\iint_{Q'_i} f \dX = 0.
$$
Once we have done that, we show that we can control $f$ on $Q_i\backslash Q'_i$ with the help of the assumption on $f$. Finally we will put it all together to conclude the lemma. \\
\noindent \emph{Step 1}. 
Fix $i$, and set $z \lequiv M_i - u$. We see that $z \geq 0$ on $Q_i$ and $z$ satisfies
\begin{equation}\label{eq:kl_z}
z_t - (a_{\alpha\beta} z_{x_\beta})_{x_\alpha} \ge f.
\end{equation}
In particular, due to parabolic Harnack inequality (see Theorem 6.24 in ~\cite{Lieberman96})
\begin{equation}\label{eq:kl_harnack_z}
\dashiint_{Q''_i} z\dX \leq C \inf_{Q_{i+1}} z.
\end{equation}
Let $w$ solve backward time parabolic equation
\begin{equation}\label{eq:kl_w}
-w_t - (a_{\alpha\beta}w_{x_\alpha})_{x_\beta} = \frac{1}{R^{2}_i}\chi_{Q''_{i}}
\end{equation}
on $Q_i$ with $w\equiv 0$ on the backward time parabolic
boundary, that is
$$
w \equiv 0 \textrm{ on }(\partial \B_{R_i}(x_0) \times [t_0-R^2_{i}, t_0])\cup
(B_{R_i}(x_0)\times \set{t_0}).
$$
At this point if $zw$ were actually differentiable in time, we
would multiply equation \eqref{eq:kl_w} by $zw$ and
integrate by parts to obtain
$$
\iint_{Q_i} -z\left(\frac{w^2}{2}\right)_t + a_{\alpha\beta}
z_\beta\left(\frac{w^2}{2}\right)_{x_\alpha}  + a_{\alpha\beta}
w_{x_\alpha}w_{x_\beta}z \dX = \frac{1}{R_i^2}\iint_{Q''_{i}} zw\dX,
$$
and since $z$ satisfies equation \eqref{eq:kl_z} we
would conclude that
\begin{equation}\label{eq:key_lemma_sec_to_last}
\frac{1}{R_i^n}\iint_{Q'_i} f\left(\frac{w^2}{2}\right)\dX \leq
\frac{1}{R_i^{n+2}}\iint_{Q''_{i}} zw\dX.
\end{equation}
In general we cannot expect $zw$ to be differentiable in time.
To obtain equation \eqref{eq:key_lemma_sec_to_last} rigorously
one would need to use Steklov average
$$
(zw)^h(x,t):=\frac{1}{h}\int_{t-h}^t z(x,\tau)w(x,\tau)\bd \tau
$$
as a test function in \eqref{eq:kl_w}. However, we will not do this
here, instead I refer the reader to Lemma 6.1 in
~\cite{Lieberman96}, where similar computation has been carried out.

Now, since $w$ solves \eqref{eq:kl_w}, by strong maximal
principle $w \ge \theta >0$ on $Q'_i$ and also $w \le C$
on $Q_i$, with bounds independent of $i$ (one can see this by
scaling for example). Therefore, combining this observation with
inequality \eqref{eq:kl_harnack_z}, we obtain
$$
\frac{1}{R_i^n}\iint_{Q'_i} f \dX \le C\dashiint_{Q''_{i}} z \dX \le C \inf_{Q_{i+1}} z = C (M_i - M_{i+1}).
$$
However, since
$$
\sum_{i=0}^\infty M_i - M_{i+1} \leq \sup_{Q_0} u,
$$
we conclude that
$$
\frac{1}{R_i^n}\iint_{Q'_i} f \dX\To 0 \quad \textrm{as}\quad i \To \infty.
$$
\\
\noindent\emph{ Step 2}.
Let $A$ be some measurable set. We will show that for all $Q_R\subset Q_{8R} \subset Q_{R_0}$ and for all $\epsilon$, there exists $\delta$ such that if $|A\cap Q_R|\leq \delta |Q_R|$, then
$$
\frac{1}{R^n}\iint_{A\cap Q_R}f \dX \leq \epsilon.
$$
First, we can easily deduce by the argument similar to the one in part 1, that
\begin{equation}\label{eq:kl_bound_f}
\frac{1}{(4R)^n}\iint_{Q_{4R}}f \dX \le C,
\end{equation}
where constant is independent of $R$.

We now use \eqref{eq:higher_Lp} to conclude that
\begin{eqnarray*}
\iint_{A\cap Q_R} f\dX & \le & \left(\frac{1}{|Q_R|}\iint_{A\cap Q_R} f^\alpha \dX \right)^{1/\alpha}|Q_R|^{1/\alpha}|A\cap Q_R|^{1-1/\alpha} \\
& = & C|Q_R|^{1/\alpha}|A\cap Q_R|^{1- 1/\alpha} \dashiint_{Q_{4R}} f \dX \\
& = & C\left(\frac{|A\cap Q_R|}{|Q_R|}\right)^{1-1/\alpha} \iint_{Q_{4R}}f \dX \leq C_1R^n\delta^{1-1/\alpha}.
\end{eqnarray*}
Above, the last inequality follows by \eqref{eq:kl_bound_f}.

\noindent \emph{Step 3}.\\
Finally we put everything together. Fix $\epsilon >0$. First notice that by choosing $\sigma$ small enough we can make
$$
|Q_i\backslash Q'_i| \leq \delta |Q_i|,
$$
where $C_1\delta^{1-1/\alpha} < \epsilon /2$. Then by first step we can find $i >2$ such that
$$
\frac{1}{R_i^n}\iint_{Q'_i} f \dX \leq \epsilon/2.
$$
Finally, the above together with conclusion of second step gives us
$$
\frac{1}{R_i^n}\iint_{Q_i} f \dX = \frac{1}{R_i^n}\iint_{Q'_i} f \dX + \frac{1}{R_i^n}\iint_{Q_i\backslash Q'_i} f \dX < \epsilon.
$$
\qed\end{proof}

Now we are in position to apply our key Lemma to deduce crucial importance of entropy in questions of everywhere regularity for parabolic systems.
\begin{theorem}[Entropy condition]\label{thm:entropy_condition}
Let $\vu$ be weak solution of \eqref{eq:nonlinear_system}. Suppose there exists $\phi\in \V(\pDom)$ that together with $\vu$ weakly satisfy the following inequality:
$$
\phi_t - (a_{\alpha\beta}\phi_{x_\beta})_{x_\alpha} + \lambda|\Grad \vu|^2 \leq 0,
$$
where $a_{\alpha\beta}$ are bounded and strictly elliptic. Then $\vu$ is everywhere H\"older continuous on the interior of $\pDom$.
\end{theorem}
\begin{proof}
Since $\vu$ satisfies condition \eqref{eq:Lp_estimate}, we see immediately that conditions of the key Lemma \ref{lem:key_lemma} are satisfied. Therefore, we conclude that
$$
\liminf_{R\To 0}\frac{1}{R^n}\iint_{Q(X_0,R)} |\Grad \vu|^2 \dX= 0.
$$
However, this is precisely the condition \eqref{eq:local_regularity_condition} of Theorem \ref{thm:local_regularity}. Hence we conclude that $\vu$ is everywhere H\"older continuous on the interior of $\pDom$.
\qed\end{proof}

\section{Everywhere regularity of certain diffusion systems}
In this section we come back to the discussion of weak solutions of the diffusion system \eqref{eq:gen_diff_eq}, imposing additional requirement that $\vPhi(z)$ is a function of only the norm of $z$. As we will see this allows us to conclude much more about solutions of this equation. In particular solutions are bounded if they are bounded initially and at the boundary. More importantly, solutions are actually everywhere H\"older continuous and thus smooth if $\vPhi$ is.

First I will show that in the case when $\vPhi$ is a functions of only the norm, weak solutions of equation \eqref{eq:gen_diff_eq} that are bounded initially and bounded at the boundary will remain bounded for all time.
\begin{theorem}[Boundedness]\label{thm:boundedness}
Let $\vu$ be a weak solution of the generalized diffusion equation \eqref{eq:gen_diff_eq}. Also suppose that $\norm{\vu(\cdot,0)}_{\Lp{\infty}(\eDom)}< \infty$ and $\norm{\vu(\cdot,t)}_{\Lp{\infty}(\partial\eDom)} < \infty$ for all $t\in [0,T]$. Then $\norm{\vu}_{\Lp{\infty}(\eDom)} < \infty$ and we have
\begin{equation}\label{eq:sup_bound}
\norm{\vu(\cdot, t)}_{\Lp{\infty}(\eDom)}\leq \max\set{\norm{\vu(\cdot,0)}_{\Lp{\infty}(\eDom)}, \sup_{s\in [0,t]}\norm{\vu(\cdot,s)}_{\Lp{\infty}(\partial\eDom)} }.
\end{equation}
\end{theorem}
\begin{proof}
Set $B$ to
$$
B = \max\set{\norm{\vu(\cdot,0)}_{\Lp{\infty}(\eDom)}, \sup_{s\in [0,t]}\norm{\vu(\cdot,s)}_{\Lp{\infty}(\partial\eDom)}}.
$$
Fix $\epsilon > 0$ that is less than one. Let $\gamma:\Real_+\To\Real$ be a smooth convex function which is identically zero on $[0,B+\epsilon]$, positive and increasing otherwise, and linear on $[B+1, \infty)$. Also set $\Gamma(z)$ to $\gamma(|z|)$. Then we compute
\begin{eqnarray*}
\frac{d}{dt}\int_\eDom \Gamma(\vu(x,t))\dx & = & \int_\eDom \Gamma_{z_i}\u^i_t\dx \\
& = & \int_\eDom \left(\Gamma_{z_i}\vPhi_{z_i z_j} \u^j_{x_\alpha}\right)_{x_\alpha}
- \Gamma_{z_i z_k}\vPhi_{z_k z_j} \u^i_{x_\alpha}\u^j_{x_\alpha} \dx \\
& \leq & \int_{\partial\eDom} \Gamma_{z_i}\vPhi_{z_i z_j} \u^j_{x_\alpha} \nu_\alpha \dS \\
& = & 0.
\end{eqnarray*}
The inequality is true because Hessians of two functions of only the norm commute and both $\Gamma$ and $\vPhi$ are convex. The last equality is true because $\gamma$ is identically zero on $[0,B+\epsilon]$ and $|\vu|$ is less than or equal to $B$ on the boundary. Since $\Gamma(\vu)$ is positive and initially zero we conclude that $\Gamma$ is zero up to time $t$ and thus $\norm{\vu(\cdot, t)}_\infty \leq B +\epsilon$. Since the inequality is true for all $\epsilon >0$ the Theorem follows.
\qed\end{proof}

The next Lemma will show that if we suppose that $\vPhi(z)$ is a function only of the norm of $z$, then there exists an entropy that satisfies conditions of Theorem \ref{thm:entropy_condition}.
\begin{lemma}\label{lem:gen_diff_scalar}
Let $\vu$ be a weak bounded solution of \eqref{eq:gen_diff_eq} and suppose
$\vPhi$ is of the form
$$
\vPhi(\vu) = \vphi(|\vu|).
$$
Then there is a continuously differentiable, strictly increasing function
$\gamma:\Real\To\Real$ such that $\vphi=\vphi(|\vu|)$ weakly satisfies the following
inequality:
\begin{equation}
\vphi_t - \Laplace(\gamma(\vphi)) + \lambda^2 |\Grad\vu|^2 \leq 0.
\end{equation}
\end{lemma}
\begin{proof}
First of all, without loss of generality we can suppose that $\vPhi(0)=0$.
Notice that $\vphi$ satisfies the following equation weakly:
$$
\vphi_t - (\vPhi_{z_i}\vPhi_{z_i z_j}\u^j_{x_\alpha})_{x_\alpha} +
|\Grad_x(\Grad\vPhi(\vu))|^2=0.
$$
Looking at the quantity inside the divergence term we see that it is equal to
$$
\vPhi_{z_i}\vPhi_{z_i z_j}\u^j_{x_\alpha} =
\left(\frac{1}{2}|\Grad\vPhi(\vu)|^2\right)_{x_\alpha}.
$$
Since $\vPhi(z) = \vphi(|z|)$, we observe that
$$
\frac{1}{2}|\Grad \vPhi(\vu)|^2 = \frac{1}{2} \vphi'(|\vu|)^2.
$$
Let $\vpsi$ be the continuous inverse of $\vphi$. Set
$$
\gamma(z) = \int_0^z \vphi''(\vpsi(t)) \dt.
$$
Then multiplying $\gamma'(\phi(z))$ by $\phi'(z)$ and integrating, we see that $\gamma$ and $\vphi$ satisfy
$$
\gamma(\vphi(z)) = \frac{1}{2}\vphi'(z)^2.
$$
This $\gamma$ is continuously differentiable and strictly
increasing, since $\vphi$ is strictly convex. Therefore we
conclude that $\vphi$ satisfies
$$
\vphi_t - \Laplace(\gamma(\vphi))+|\Grad_x(\Grad\vPhi(\vu))|^2=0,
$$
and due to strict convexity the last term on the left hand side is greater or
equal to $\lambda^2 |\Grad \vu|^2$.
\qed\end{proof}

Now we are in position to use the above Lemma \ref{lem:gen_diff_scalar}
together with Theorem \ref{thm:entropy_condition} to deduce
\begin{theorem}
Weak solutions of equation (\ref{eq:gen_diff_eq}) are H\"older
continuous in $\pDom$.
\end{theorem}
\begin{proof}
Lemma \ref{lem:gen_diff_scalar} tells us, that for $\vu$ a weak bounded solution of \eqref{eq:gen_diff_eq}, there exists $\phi$ satisfying
$$
\vphi_t - (a\vphi_{x_\alpha})_{x_\alpha} + \lambda^2 |\Grad\vu|^2 \leq 0,
$$
where $a(x,t) \lequiv \gamma'(\vphi(x,t))$. However, this is precisely the condition of Theorem \ref{thm:entropy_condition}. Therefore, we conclude the proof.
\qed\end{proof}

\section{Strongly coupled parabolic systems}
One of the earliest nontrivial examples of quasi-linear
parabolic systems whose solutions have interior everywhere
regularity was due to Wiegner ~\cite{Wiegner92}. These are the
so-called strongly coupled parabolic systems of the following
form:
\begin{equation}\label{eq:strongly_coupled}
\u^i_t - (a_{\alpha\beta}\u^i_{x_\beta} +
c^i_{\alpha\beta}H_{x_\beta})_{x_\alpha} = 0,
\end{equation}
where
\begin{enumerate}
\item $H = H(\vu)$ is a function of $\vu$;
\item $A^{ij}_{\alpha\beta}:= a_{\alpha\beta}\delta^{ij} + c^i_{\alpha\beta}H_{z_j}$ and $a_{\alpha\beta}$
are strictly elliptic in the sense that
$$
 \lambda |\xi|^2
\le A^{ij}_{\alpha\beta}\xi^i_\alpha\xi^j_\beta,\textrm{ and } \lambda |\zeta|^2 \le a_{\alpha\beta} \zeta^i \zeta^j, \quad \textrm{for all }\xi\in \Real^{nN}, \zeta\in \Rn;
$$
\item $H(z)$ is twice continuously differentiable and
$\lambda |\zeta|^2 \le  H_{z_i z_j} \zeta^i \zeta^j;$
\item $a_{\alpha \beta}$ and $c^i_{\alpha\beta}$ are bounded.
\end{enumerate}

After Wiegner, Dung ~\cite{Dung99} also worked on these types of system. However, neither Wiegner's nor Dung's proofs of everywhere regularity for solutions of strongly coupled parabolic systems reduce to something analogous to the key Lemma \ref{lem:key_lemma}. It is instructive to prove everywhere regularity of weak solutions to \eqref{eq:strongly_coupled} using the key Lemma to illustrate an underlying similarity between strongly coupled parabolic systems and generalized diffusion equations \eqref{eq:gen_diff_eq}. It appears that for both systems discussed in this paper the existence of an \emph{entropy} is crucial for everywhere regularity of their solutions.

\begin{remark}
When $\vPhi$ only depends on the norm of the gradient, diffusion system \eqref{eq:gen_diff_eq} actually has the form of a strongly coupled system, except with possibly \emph{non-convex} $H$. Indeed, if $\vPhi(z)= \vphi(|z|)$, then
$$
\vPhi_{z_i z_j}(z) = \frac{\vphi'(|z|)}{|z|}\delta_{ij} +
\frac{z_i}{|z|}\left(\vphi''(|z|) - \frac{\vphi'(|z|)}{|z|}\right)\frac{z_j}{|z|},
$$
therefore, with
$$
a_{\alpha\beta} = \frac{\vphi'(|\vu|)}{|\vu|}\delta_{\alpha\beta}, \quad
c^i_{\alpha\beta} =\frac{\u^i}{|\vu|}\delta_{\alpha\beta},
$$
and
$$
H(z) = \vphi'(|z|) - \int_0^{|z|} \frac{\vphi'(s)}{s} \ds,
$$
system \eqref{eq:gen_diff_eq} has the form \eqref{eq:strongly_coupled}.
\end{remark}

We can prove an interior everywhere regularity result for strongly coupled
parabolic systems rather easily using Theorem \ref{thm:entropy_condition}.
As in the previous section I will show existence of an entropy.
\begin{lemma}
Let $\vu \in \V(\pDom; \RN)$ be a weak bounded solution of
(\ref{eq:strongly_coupled}), then for some large enough $s$ there is a positive
constant $c$ such that $v := e^{sH}$ is a subsolution of the following
equation:
\begin{equation}\label{eq:strongly_coupled_scalar}
v_t - (A_{\alpha\beta} v_{x_\beta})_{x_\alpha} + c|\Grad\vu|^2 \leq 0.
\end{equation}
\end{lemma}
\begin{proof}
We compute
\begin{eqnarray*}
(e^{sH})_t  &=& se^{sH}H_{z_i}(a_{\alpha\beta}\u^i_{x_\beta} +
c^i_{\alpha\beta}H_{x_\beta})_{x_\alpha} \\
& = & (A_{\alpha\beta}(e^{sH})_{x_\beta})_{x_\alpha} -
(se^{sH}H_{z_i})_{x_\alpha}(a_{\alpha\beta}\u^i_{x_\beta} +
c^i_{\alpha\beta}H_{x_\beta}).
\end{eqnarray*}
The last term on the right becomes
\begin{multline*}
(se^{sH}H_{z_i})_{x_\alpha}(a_{\alpha\beta}\u^i_{x_\beta}
+c^i_{\alpha\beta}
H_{x_\beta})  =  s^2 e^{sH}A_{\alpha\beta}H_{x_\alpha} H_{x_\beta}  \\
+ se^{sH}H_{z_i z_j}a_{\alpha\beta}\u^i_{x_\alpha}\u^j_{x_\beta} +
se^{sH}H_{z_i z_j}c^i_{\alpha\beta}H_{x_\beta}\u^i_{x_\alpha} \\
\ge se^{sH}\left( \lambda s|\Grad H|^2 + \lambda |\Grad \vu|^2
-C(\epsilon)|\Grad H|^2 - \epsilon|\Grad \vu|^2\right) \ge
\frac{\lambda}{2} se^{sH}|\Grad \vu|^2.
\end{multline*}
The last inequality follows by first making $\epsilon$ small and then $s$ large. Since
$\vu$ is bounded, we have $H$ is bounded from below. Therefore, for some
$c$, $v$ satisfies equation (\ref{eq:strongly_coupled_scalar}) as claimed.
\qed\end{proof}
At this point we immediately conclude that conditions of Theorem \ref{thm:entropy_condition} are satisfied. Therefore we have established
\begin{theorem}
Bounded weak solutions of (\ref{eq:strongly_coupled}) are
H\"older continuous in $\pDom$.
\end{theorem}

\bibliography{evreg}

\providecommand{\bysame}{\leavevmode\hbox to3em{\hrulefill}\thinspace}
\providecommand{\MR}{\relax\ifhmode\unskip\space\fi MR }
\providecommand{\MRhref}[2]{%
  \href{http://www.ams.org/mathscinet-getitem?mr=#1}{#2}
}
\providecommand{\href}[2]{#2}
\begin{thebibliography}{1}

\bibitem{BaeChoe04}
Hyeong-Ohk Bae and Hi~Jun Choe, \emph{Regularity for certain nonlinear
  parabolic systems}, Comm. Partial Differential Equations \textbf{29} (2004),
  no.~5-6, 611--645. \MR{MR2059143 (2005d:35105)}

\bibitem{Dung99}
Le~Dung, \emph{H\"older regularity for certain strongly coupled parabolic
  systems}, J. Differential Equations \textbf{151} (1999), no.~2, 313--344.
  \MR{MR1669717 (2001a:35033)}

\bibitem{Giaquinta83}
Mariano Giaquinta, \emph{Multiple integrals in the calculus of variations and
  nonlinear elliptic systems}, Annals of Mathematics Studies, vol. 105,
  Princeton University Press, Princeton, NJ, 1983. \MR{MR717034 (86b:49003)}

\bibitem{GiaquintaStruwe82}
Mariano Giaquinta and Michael Struwe, \emph{On the partial regularity of weak
  solutions of nonlinear parabolic systems}, Math. Z. \textbf{179} (1982),
  no.~4, 437--451. \MR{MR652852 (83f:35062)}

\bibitem{Lieberman96}
Gary~M. Lieberman, \emph{Second order parabolic differential equations}, World
  Scientific Publishing Co. Inc., River Edge, NJ, 1996. \MR{MR1465184
  (98k:35003)}

\bibitem{Uhlenbeck77}
K.~Uhlenbeck, \emph{Regularity for a class of non-linear elliptic systems},
  Acta Math. \textbf{138} (1977), no.~3-4, 219--240. \MR{MR0474389 (57
  \#14031)}

\bibitem{Wiegner92}
Michael Wiegner, \emph{Global solutions to a class of strongly coupled
  parabolic systems}, Math. Ann. \textbf{292} (1992), no.~4, 711--727.
  \MR{MR1157322 (92m:35136)}

\end{thebibliography}
\bibliographystyle{amsplain}

\end{document}